\documentclass[12pt]{amsart}
\usepackage[top=30truemm,bottom=30truemm,left=25truemm,right=25truemm]{geometry}
\usepackage{mathrsfs}
\usepackage{tikz-cd}
\usepackage{color}
\usepackage{bm}
\usepackage{amsfonts,amssymb}
\usepackage{dsfont}
\usepackage{amscd}
\usepackage{extarrows}
\usepackage{amsmath}
\usepackage{mathrsfs}
\usepackage{enumerate}
\usepackage{amscd}
\usepackage[all]{xy}
\usepackage{autobreak}
\usepackage[pagebackref,colorlinks]{hyperref}
\theoremstyle{plain} 
\newtheorem{theorem}{\indent\bf Theorem}[section]

\theoremstyle{definition} 

\newtheorem{definition theorem}[theorem]{\indent\bf Definition-Theorem}

\newtheorem{problem}[theorem]{\indent\bf Problem}

\newtheorem{thm}{Theorem}[section]
\newtheorem{cor}[thm]{Corollary}
\newtheorem{lem}[thm]{Lemma}
\newtheorem{prop}[thm]{Proposition}

\theoremstyle{definition}
\newtheorem{defn}{Definition}[section]

\theoremstyle{remark}
\newtheorem{rem}{Remark}[section]

\newcommand{\bi}{\begin{itemize}}
\newcommand{\ei}{\end{itemize}}
\newcommand{\be}{\begin{equation}}
	\newcommand{\ee}{\end{equation}}
\newcommand{\bea}{\begin{eqnarray}}
	\newcommand{\eea}{\end{eqnarray}}
\newcommand{\ben}{\begin{eqnarray*}}
	\newcommand{\een}{\end{eqnarray*}}
\newcommand{\bt}{\begin{split}}
	\newcommand{\et}{\end{split}}
\newcommand{\bet}{\begin{equation}}
	\newcommand{\mc}{\mathbb{C}}
	
	\newcommand{\ra}{\rightarrow}
\newcommand{\h}{\textbf}
\newcommand{\tm}{\text{mult}}
\begin{document}
		\title[]{Multiplicities and modifications, and singularities associated to blowing down negative vector bundles}
		
		\author[F. Deng]{Fusheng Deng}
		\address{Fusheng Deng: \ School of Mathematical Sciences, University of Chinese Academy of Sciences\\ Beijing 100049, P. R. China}
		\email{fshdeng@ucas.ac.cn}
		

	\author[Y. Li]{Yinji Li}
	\address{Yinji Li:  Institute of Mathematics\\Academy of Mathematics and Systems Sciences\\Chinese Academy of
		Sciences\\Beijing\\100190\\P. R. China}
	\email{1141287853@qq.com}

           \author[Q. Liu]{Qunhuan Liu}
           \address{Qunhuan Liu:  Institute of Mathematics\\Academy of Mathematics and Systems Sciences\\Chinese Academy of
		Sciences\\Beijing\\100190\\P. R. China}

          \email{liuqunhuan23@mails.ucas.edu.cn}
		\author[X. Zhou]{Xiangyu Zhou}
		\address{Xiangyu Zhou: Institute of Mathematics\\Academy of Mathematics and Systems Sciences\\and Hua Loo-Keng Key
			Laboratory of Mathematics\\Chinese Academy of
			Sciences\\Beijing\\100190\\P. R. China}
		\address{School of
			Mathematical Sciences, University of Chinese Academy of Sciences,
			Beijing 100049, P. R. China}
		\email{xyzhou@math.ac.cn}
		
		\begin{abstract}
We first present the mixed Hilbert-Samuel multiplicities of analytic local rings over $\mc$ as generalized Lelong numbers and further represent them as intersection numbers in the context of modifications.
As applications, we give estimates or an exact formula for the multiplicities of isolated singularities that given by the Grauert blow-downs of negative holomorphic vector bundles. 
		\end{abstract}
		
		\thanks{}

		\maketitle
		\tableofcontents

\section{Introductions}
The present work is originally motivated by the consideration of isolated singularities given by the blow-downs
of negative holomorphic vector bundles over compact complex manifolds.

Let $p:E\ra M$ be a holomorphic vector bundle over a compact complex manifold $M$.
We call $E$ \h{negative} (in the sense of Grauert) if its dual bundle $E^*$ is ample.
By a fundamental result of Grauert \cite{Gr62}, $E$ is negative if and only if the zero section, 
identified with $M$, in $E$ is exceptional, namely,
there is a complex space $A$ and a proper holomorphic surjective map $\pi:E\ra A$ such that 
$\pi(M)=\{a\}$ is a single point in $A$ and $\pi:E\backslash M\ra A\backslash\{a\}$ is biholomorphic.
We will call $(A,a)$ the \h{Grauert blow down} of $E$, which is unique up to normalization.
In this paper, we assume that a Grauert blow down is always normal.

In general $a$ is a singularity of $A$. Indeed, it is shown in our previous work that $A$ is regular at 
$a$ if and only if $M$ is biholomorphic to a projective space $\mathbb P^n$ and $E$ is isomorphic to 
the tautological line bundle $\mathcal O_{\mathbb P^n}(-1)$ over $\mathbb P^n$  as holomorphic vector bundles \cite{DLLZ24}. 
In the present paper, we are interested in further study of the singularity property of $A$ at $a$.

For a general complex space $Y$ of pure dimension $n$ with $y\in Y$,
a fundamental invariant that reflects the singularity property of $Y$ at $y$ is the multiplicity of $Y$ at $y$.
There are various different formulations of the multiplicity.
We list three of them as follows:
\bi
\item[(1)] {\emph{ Geometric formulation $\text{mult}(Y,y)$}}: embed $Y$ into some $\mc^N$ locally near $y$, then the projection from some small neighborhood $U$ of $y$ in $Y$ to a generic $n$-dimensional linear subspace in $\mc^N$ is a ramified covering.
Then $\text{mult}(Y,y)$ is defined to be the degree of this ramified covering.
\item[(2)] { \emph{Algebraic formulation $\text{e}(\mathcal O_{Y,y}, \mathfrak M)$}}: the Hilbert-Samuel multiplicity $\text{e}(\mathcal O_{Y,y}, \mathfrak M)$ , where $\mathfrak M$ is the maximal ideal of the local ring $\mathcal O_{Y,y}$ of germs of holomorphic functions on $Y$ at $y$.
\item[(3)] { \emph{Analytic formulation $\nu([Y],y)$}}: embed $Y$ into some $\mc^N$ locally near $y$ and get a positive closed current $[Y]$ on some open set of $\mc^N$, then $\nu([Y],y)$ is the Lelong number of $[Y]$ at $y$.
\ei 

The definitions of $\text{mult}(Y,y)$ and $\nu([Y],y)$, that we will recall later, seems not intrinsic for $(Y,y)$ at the first glance, 
since the definitions depend on the embedding of $Y$ into $\mc^N$ locally.
However, it turns out that all the above formulations are equivalent, namely, we have the following fundamental equalities:
$$\text{e}(\mathcal O_{Y,y}, \mathfrak M)=\tm(Y,y)=\nu([Y],y).$$
In particular these invariants (equal to each other) are all intrinsic, i.e., invariant under biholomorphic transformations.
The first equality can be proved based on some ideas of Chevally and Samuel (see \cite{Mum76}),
and the second one is a famous result of Thie \cite{Thi67}.

To study multiplicities of singularities associated to negative vector bundles, 
we indeed need to study multiplicities of singularities in more general contexts. 
Our first aim is to represent the mixed Hilbert-Samuel multiplicity in terms of generalized Lelong numbers,
which lays the foundation in our later study of the multiplicities the Grauert blow-downs of negative holomorphic vector bundles.

We first recall the definition of the mixed Hilbert-Samuel multiplicity.
Let $\mathcal O$ be the local ring of the holomorphic functions on a reduced complex space germ $(Y,y)$ with maximal ideal $\mathfrak M$.
An ideal $\mathfrak U$ of $\mathcal O$ is called \h{{primary to $\mathfrak M$}} if $y$ is the unique common zero of all elements in $\mathfrak U$,
or equivalently, $\mathfrak M^r\subset\mathfrak U$ for $r>>1$.

\begin{definition theorem}[see \cite{Laz04}]\label{def:multiplicity mixed}
Given positive integer $k\leq n$, $\mathfrak{M}$-primary ideals $\mathfrak{U}_1,\cdots,\mathfrak{U}_k$ of $\mathcal O$, and non-negative integers $d_1,\cdots,d_k$ such that $\sum_{j=1}^k d_j=n$, we define the (Hilbert-Samuel) mixed multiplicity
\begin{align*}
e(\mathfrak{U}_1^{[d_1]};\cdots;\mathfrak{U}_k^{[d_k]})\in\mathbb{N}
\end{align*}
by the property that 
\begin{align*}
\mathrm{dim}_{\mc}(\mathcal{O}/\mathfrak{U}_1^{t_1}\cdot...\cdot\mathfrak{U}_k^{t_k})
\end{align*} 
is given for $t_j>>1$ by a polynomial $P(t_1,\cdots,t_k)$ of the form
\begin{align*}
P(t_1,\cdots,t_k)=\sum_{d_1+\cdots+d_k=n}\frac{n!}{d_1!\cdot...\cdot d_k!}e(\mathfrak{U}_1^{[d_1]};\cdots;\mathfrak{U}_k^{[d_k]})\cdot t_1^{d_1}\cdot...\cdot t_k^{d_k}+\mathrm{lower}\ \mathrm{order} \ \mathrm{terms}.
\end{align*}
When $k=1$, $e(\mathfrak{U}^{[n]})=:e(\mathfrak{U})$ is the classical Hilbert-Samuel multiplicity of $\mathfrak{U}$ in $\mathcal{O}$.
\end{definition theorem}
Suppose that $Y\subseteq \mathbb{C}^M$ is an analytic subset of pure dimension $n$.
We denote the current of integration along $Y$ by $[Y]$, which is a closed positive current of bidimension $(n,n)$.
The Lelong number of $[Y]$ at $0$ is given by
\begin{align*}
\nu([Y],0):=\lim_{r\ra0}\int_{\{|z|<r\}}[Y]\wedge (dd^c\log|z|)^n.
\end{align*}
The limit in the right hand side can be understood as the mass of the origin $0$ 
with respect to the measure $[Y]\wedge (dd^c\log|z|)^n$ on $\mc^M$.
In the present work, we consider generalized Lelong number of $[Y]$ with respect to mixed weight functions.
\begin{defn}[see \cite{Dem}]
Suppose $g_j=(g_{j,1},\cdots,g_{j,N_j})$, $j=1,\cdots,n$, are $n$ tuples of holomorphic function germs on $(\mathbb{C}^M,0)$ such that $\{g_{j,1}=\cdots=g_{j,N_j}=0\}\cap Y=\{0\}$. 
The \h{generalized Lelong number} of $[Y]$ with respect to the mixed weights $(g_1,\cdots,g_n)$ is defined as
\begin{align*}
\nu([Y],0;\log|g_1|,\cdots,\log|g_n|)=\lim_{r\ra0}\int_{\{|z|<r\}}[Y]\wedge(dd^c\log|g_1|)\wedge\cdots\wedge(dd^c\log|g_n|),
\end{align*}
where $\log|g_j|:=\frac{1}{2}\log(|g_{j,1}|^2+\cdots+|g_{j,N_j}|^2)$.
\end{defn}

Similarly, $\nu([Y],0;\log|g_1|,\cdots,\log|g_n|)$ can be understood as the mass of the origin $0$ 
with respect to the measure $[Y]\wedge(dd^c\log|g_1|)\wedge\cdots\wedge(dd^c\log|g_n|)$ on $\mc^M$.
The condition that $\{g_{j,1}=\cdots=g_{j,N_j}=0\}\cap Y=\{0\}\ (1\leq j\leq n)$ ensures that the product of the involved currents is well defined.

Our first result represents the mixed Hilbert-Samuel multiplicities in terms of generalized Lelong numbers.
\begin{thm}\label{thm:mult vs lelong}
Let $(Y,y)\subset(\mc^M,0)$ be a germ of $n$-dimensional irreducible reduced analytic subset, 
$\mathfrak{M}$ be the maximal ideal in $\mathcal{O}_{Y,y}$, and $\mathfrak{U}_1,\cdots,\mathfrak{U}_n \subset \mathcal{O}_{Y,y}$ 
be $\mathfrak{M}$-primary ideals.
Assuming that $\mathfrak{U}_j$ is generated by $g_{j,1},\cdots,g_{j,N_j} \in\mathcal{O}_{\mc^M,0}$, then it holds that 
\begin{align*}
e(\mathfrak{U}_1;\cdots;\mathfrak{U}_{n})=\nu([Y],0;\log|g_1|,\cdots,\log|g_n|).
\end{align*}
\end{thm} 
From Theorem \ref{thm:mult vs lelong}, we infer that $\nu([Y],0;\log|g_1|,\cdots,\log|g_n|)$ is intrinsic. 
Indeed, this also follows from the comparison theorem of Lelong numbers directly, see Lemma \ref{lem:comparison thm}.

In the case that $\mathfrak{U}_1=\cdots=\mathfrak{U}_n=\mathfrak{M}$, 
we get 
$$e(\mathfrak{M})=\nu([Y],0)=\nu([Y],0;\log|z|,\cdots, \log|z|),$$
which recovers the result of Thie mentioned above.

If $(Y,y)$ is regular and $\mathfrak{U}_1=\cdots=\mathfrak{U}_n$, the above theorem is proved by Demailly in \cite{Dem09}.
Our approach is quite different from that of Demailly.
In our argument, Samuel's projection fomula and transformation fomula of Lelong numbers under ramified coverings play essential roles.

With Theorem \ref{thm:mult vs lelong} in hand, we can now study relations between multiplicity of singularities and modifications.
The main aim is to transform Hilbert-Samuel multiplicities into intersection numbers and Monge-Amp\`ere products by a proper modification.

Let us first recall some basic notions.

\begin{defn}\label{def:modification}
A proper surjective holomorphic map $\pi:X\ra Y$ of reduced and irreducible complex spaces is called a \h{modification} if there are analytic subsets $A\varsubsetneqq X$ and $B\varsubsetneqq Y$
such that:
\begin{itemize}
\item[(1)] $B=\pi(A)$,
\item[(2)] $\pi|_{X\backslash A}:X\backslash A\ra Y\backslash B$ is biholomorphic.
\end{itemize}
\end{defn}
By a modification $\pi:(X,A)\ra (Y,B)$, we always refer to a modification of the form as in Definition \ref{def:modification}.


\begin{defn}\label{def:desingularization}
A modification $\pi:(X,A)\ra (Y,B)$ is called a \h{desingularization} of $(Y, B)$ if $X$ is regular.
\end{defn}

Let us first recall a fundamental result of Ramanujam:

\begin{thm}\label{thm:Ramanujam}[\cite{Ram73}\cite{Le09}]
If $\pi:(X, E)\ra (Y,y)$ is a desingularization such that the ideal subsheaf in $\mathcal O_X$ generated by $\pi^{*}\mathfrak M_{Y,y}$ is invertible (locally generated by a single element), 
then $\tm(Y,y)$ can be represented as the self-intersection number of $E$ as 
$$\tm(Y,y)=(-1)^{n-1}k^nE^n,$$
where $k$ is the vanishing order of $\pi^{*}\mathfrak M_{Y,y}$ along $E$.
\end{thm}

Theorem \ref{thm:Ramanujam} in the algebraic geometry setting was proved by Ramanujam, and in the analytic setting 
was proved by Tráng, following the idea of Ramanujam.

Here we should emphasize that, in Theorem \ref{thm:Ramanujam}, 
the condition on the invertibility of the ideal subsheaf in $\mathcal O_X$ generated by $\pi^{*}\mathfrak M_{Y,y}$
is indispensable. 
On the other hand, this condition is not satisfied for general desingularization.
For example,  the Grauert blow down $\pi:L\ra (Y,y)$ of a negative line bundle $L$ satisfies this condition 
if and only if $(L^*)^k$ is globally generated, where $k$ is the minimal integer such that $(L^*)^k$ has nonzero global holomorphic section (see Proposition \ref{prop:L globally generated}).
For this reason, to study the multiplicity of the Grauert blow downs of general negative vector bundles, 
we need to try to generalize Theorem \ref{thm:Ramanujam} to the case of general modifications,
without assuming that the associated inverse ideal subsheaf is invertible.

In the effort to do this, for a given desingularization $\pi:(X, E)\ra (Y,y)$ and $\mathfrak{M}$-primary ideals $\mathfrak U_j\subset\mathcal O_{Y,y}$,
we try to define some natural related objects that reflect the deviation for the ideal subsheaf of $\mathcal O_X$ generated by $\pi^{*}\mathfrak U_j$ from being invertible.
We assume that $E$ is a hypersurface in $X$, i.e. $E$ has codimension one.
In the following sequel, we always assume $\mathfrak{U}_j$ is generated by $g_j=(g_{j,1},\cdots,g_{j,N_j})$, $g_{j,k}\in\mathcal{O}_{Y,y}$, $j=1,\cdots,n$.
We introduce the following notations:
\begin{itemize}
\item
$E=\bigcup_{i=1}^N E_i$, $E_i$ are irreducible components of $E$,
\item
$D_j=\sum_{i=1}^N m_{i,j}E_i$, $m_{i,j}$ is the vanishing order of $\pi^*\mathfrak{U}_j$ along $E_i$,
\item
$dd^c\pi^*\log|g_j|=\sum_{i=1}^Nm_{i,j}[E_i]+T_j$ is the Siu's decomposition, $T_j$ is a closed positive $(1,1)$-current with analytic singularity $A_j\subsetneqq E$ (see the proof of Theorem \ref{thm:main thm} in \S \ref{sec:lelong vs MA} for details).
\end{itemize}

For $f\in\mathcal{O}(X)\backslash\{0\}$, we denote the vanishing order of $f$ along $E_i$ by $\mathrm{ord}_{E_i}(f)$. The vanishing order of $\pi^*\mathfrak{U}_j$ along $E_i$ is defined as $\min_{g\in\mathfrak{U}_j\backslash\{0\}}\{\mathrm{ord}_{E_i}(\pi^*g)\}$.

The analytic set $A_j\subset X$ reflects the deviation from being invertible for the ideal sheaf in $\mathcal O_X$ generated by $\pi^{*}(\mathfrak U_j)$,
and $A_j=\emptyset$ if and only if this ideal subsheaf is invertible.

\begin{thm}\label{thm:main thm}
Assume that the germ $(Y,y)$ is locally irreducible and reduced, then it holds that
\begin{align*}
e\big(\mathfrak{U}_1;\cdots;\mathfrak{U}_n\big)=(-1)^{n-1}(D_1)\cdot...\cdot(D_n)+\int_ET_1 \wedge \cdots\wedge T_n,
\end{align*}
provided that $\{A_j\}$ satisfies the dimension conditions
$$\mathrm{codim}(A_{j_1}\cap\cdots\cap A_{j_m})\geq m$$ 
for all choices of indices $j_1<\cdots<j_m$ in $\{1,\cdots,n\}$. 
\end{thm}

\begin{rem} Remarks about Theorem \ref{thm:main thm}:
\bi
\item The dimension conditions ensure that the Monge-Amp\`ere products $T_1\wedge\cdots\wedge T_n$ is a well-defined measure on $X$, and the integration $\int_ET_1 \wedge \cdots\wedge T_n$
represents the mass of $E$ with respect to this measure.
On the other hand, we indeed get a formula for $e\big(\mathfrak{U}_1;\cdots;\mathfrak{U}_n\big)$ even without this codimension condition, 
which seems deserve to study further (see the proof of this theorem).
\item If the ideal subsheaf of $\mathcal O_X$ generated by $\pi^*\mathfrak{U}_j$ is invertible, then $T_j$ is smooth and we get
\begin{align*}
e(\mathfrak{U}_1,\cdots,\mathfrak{U}_n)=(-1)^{n-1}(D_1)\cdot...\cdot(D_n).
\end{align*}
which is a generalization of Theorem \ref{thm:Ramanujam} and is already known (see \cite[Sec 1.6 B]{Laz04}, \cite[Theorem 2.1]{Le09}). 
\item Unfortunately we do not know if the first term $(-1)^{n-1}(D_1)\cdot...\cdot(D_n)$ in the left hand side of the equality is positive in general.
\ei
\end{rem}

The dimension conditions on $A_j$ in Theorem \ref{thm:main thm} always hold if $n=2$, so we have the following corollary.

\begin{cor}\label{cor:two dim}
Assume the germ $(Y,y)$ is locally irreducible and reduced and $(Y,y)$ is of dimension 2, it holds that
\begin{align*}
e\big(\mathfrak{U}_1;\mathfrak{U}_2\big)=-(D_1)\cdot(D_2)+\int_ET_1 \wedge T_2.
\end{align*}
\end{cor}

If the condition in Theorem \ref{thm:main thm} is valid, then we have the inequality
\begin{align*}
e\big(\mathfrak{U}_1;\cdots;\mathfrak{U}_n\big)\geq(-1)^{n-1}(D_1)\cdot...\cdot(D_n).
\end{align*}
On the other hand, for this inequality to be true, the codimension condition in Theorem \ref{thm:main thm}
can be replaced by the irreducibility of $E$ and local positivity of the line bundle $\mathcal{O}_X(-E)$ on $X$.

\begin{thm}\label{thm:ineq}
Assume that the germ $(Y,y)$ is locally irreducible and reduced and $E$ is irreducible, if $\mathcal{O}_X(-E)$ is semi-positive on some neighborhood of $E$ in $X$, 
then it holds that
\begin{align*}
e\big(\mathfrak{U}_1;\cdots;\mathfrak{U}_n\big)\geq(-1)^{n-1}(D_1)\cdot...\cdot(D_n)\geq 0.
\end{align*}
\end{thm}

We now consider applications of Theorem \ref{thm:main thm} to the study of the multiplicities of Grauert blow-downs of negative line bundles. 

Let $p:L \ra M$ be a holomorphic line bundle over a compact complex manifold $M$. 
As mentioned above, $L$ is ample if and only if its dual bundle $L^*$ is negative.
If $L$ is ample, the following datum encode important information of $L$:
\begin{itemize}
\item
\h{the initial order  $k_0$}: the minimal integer such that $L^{k_0}$ has non-zero global section,
\item
\h{the second order $k_1$}: the minimal integer such that $L^{k_1}$ is globally generated,
\item
\h{the third order $k_2$}: the minimal integer such that $L^{k_1}$ is very ample (will not be used),
\item
\h{the initial base locus} $B$: the base locus of $L^{k_0}$,
\item
\h{the volume $\mathrm{vol}(L)$ of $L$}:  given by $\int_Mc_1(L)^n$.
\end{itemize}

We might regard $B$ as a non-reduced complex space. 
It is defined rigorously as follows.
Let $\pi:L^*\ra (Z,z_0)$ be the normal Grauert blow down of $L^*$.
Assume $\mathfrak{M}_{Z,z_0}$ is generated by $(z_1,\cdots,z_M)$, as a current on the total space of $L^*$, $dd^c\pi^*\log|z|$ can be decomposed as
\begin{align*}
dd^c\pi^*\log|z|=k_0[M]+T,
\end{align*}
where $T$ is a closed positive current with singularity $B$.
King's fomula yields that 
\begin{align*}
\mathds{1}_B T^{p+1}=\sum_{j}\lambda_j[B_j], 
\end{align*}
where $\mathds{1}_B$ is the characteristic function of $B$ as a set, 
$p$ is the minimum of the codimensions of the irreducible components of $B$ in $M$, 
$\lambda_j\in\mathbb{N}^*$ and $B_j$ are the irreducible components of $B$ of codimension $p$ in $M$.
In the following, $B$ should be understood as the nonreduced analytic subspace in $L^*$ given by the formal sum $\sum_{j}\lambda_j B_j$.

The following result shows that the volume of $B$ can be controlled by $\mathrm{vol}(L)$:
\begin{thm}\label{thm:vol control}
With the above notations and assumptions, we have
\begin{align*}
\mathrm{vol}_B(L):=\sum_{j}\lambda_j\int_{B_j}c_1(L)^{n-p}\leq (k_0^pk_1-k_0^{p+1})\mathrm{vol}(L).
\end{align*}
\end{thm}

Using Theorem \ref{thm:main thm} and inequalities of mixed multiplicities in \cite{RS78},\cite{Tei77}, we obtain an estimate for multiplicity of the Grauert blow down of $L^*$.

\begin{thm}\label{thm:mult of (Z,z_0)}
With the above notations and assumptions, we have
\begin{align*}
&\mathrm{mult}(Z,z_0)\geq k_0^{n+1}\mathrm{vol}(L)+(n+1-p)k_0^{n-p}\mathrm{vol}_B(L),\\
&\mathrm{mult}(Z,z_0)\leq k_0^{p+1}k_1^{n-p}\mathrm{vol}(L)+k_1^{n-p}\mathrm{vol}_B(L).
\end{align*}

\end{thm}

Of course similar result can be formulated for negative vector bundles of higher rank,
but we omit it here since the presentation will involve some nonessential complexity. 

From Theorem \ref{thm:mult of (Z,z_0)}, we can see that $L^{k_0}$ is globally generated if and only if $\mathrm{mult}(Z,z_0)=k_0^{n+1}\mathrm{vol}(L)$. 
This result can be generalized to holomorphic vector bundles of higher rank.
We first recall the definition of Segre classes.
Let $E$ be a holomorphic vector bundle of rank $r$ over $M$.
Let 
$$c(E)=1+c_1(E)+c_2(E)+\cdots\in H^\bullet(M,\mathbb Z)$$
be the total Chern class of $E$.
Then totall Segre class of $E$ is defined to be 
$$s(E)=c(E)^{-1}=1+s_1(E)+s_2(E)+\cdots.$$
For example, if $E$ is a line bundle, then $s_k(E)=(-1)^kc_1(E)^k$.

\begin{cor}\label{cor:g.g case}
Let $M$ be a compact manifold of dimension $n$, $E$ be an ample vector bundle over $M$. 
Suppose $(Z,z_0)$ is the normal Grauert blow-down of the dual bundle $E^*$, then $E$ is globally generated if and only if
\begin{align*}
\mathrm{mult}(Z,z_0)=\int_Ms_n(E^*).
\end{align*}
\end{cor}

The above discussion leads to the following natural question.
\begin{problem}
Suppose $f:(X,E)\ra (Z,z_0)$ be a desingularization with $E\subset X$ being a hypersurface. 
Let $\mathcal{I}:=f^*\mathfrak{M}$ be the pull-back of the maximal ideal $\mathcal{M}$ in $\mathcal{O}_{Z,z_0}$, and $k$ be the vanishing order of $\mathcal{I}$ along $E$.
Are the following two statements equivalent?
\begin{itemize}
\item[(1)]
$\mathcal{I}$ is invertible,
\item[(2)]
$\mathrm{mult}(Z,z_0)=(-1)^{n-1}(kE)^n$.
\end{itemize}
\end{problem}

We now specialize to the case that the base space $M$ is a compact Riemannian surface.
In this case, we can get an exact formula for $\mathrm{mult}(Z,z_0)$.

Let $M$ be a compact Riemann surface and $L$ be an ample line bundle on $M$.
We introduce the following notations:
\begin{itemize}
\item[$\bullet$]
$k_0$: the minimal integer such that $L^{k_0}$ has non-zero global section,
\item[$\bullet$]
$P_1,\cdots,P_N$: base points of $L^{k_0}$,
\item[$\bullet$]
$k_1,\cdots,k_N$: the minimal integer such that $L^{k_j}$ generates $P_j$.
\item[$\bullet$]
$d_{j,k}$: the vanishing order of $H^0(M,L^{k})$ (viewed as holomorphic functions on $L^*$) at $P_j$, $j=1,\cdots,N$, $k=k_0,\cdots,k_j$.
\end{itemize}

\begin{thm}\label{thm:RS}
With the above notations and assumptions, 
suppose that $(Z,z_0)$ is the normal Grauert blow-down of $L^*$, then
\begin{align*}
\mathrm{mult}(Z,z_0)=k_0^2\mathrm{deg}(L)+\sum_{j=1}^N\lambda_j,
\end{align*}
where $\lambda_j$ is the multiplicity of ideal generated by $\{z^{d_{j,k}}w^{k-k_0}:k=k_0,\cdots,k_j\}$ in $\mathcal{O}_{\mathbb{C}_{z,w}^2,0}$.
\end{thm}

\begin{rem}
The integers $\lambda_j$ can be calculated in an explicit and definite way, due to a result of Teissier \cite{Tei88}.
\end{rem}

As an example, we consider the canonical bundle $L=K_M$ of $M$.
If the genus $g$ of $M$ is bigger than 1,  then $K_M$ is ample and globally generated, 
so we can deduce from either Corollary \ref{cor:g.g case} or Theorem \ref{thm:RS} that 
$$\mathrm{mult}(Z,z_0)=2g-2,$$ 
which is a result proved by Morrow and Rossi \cite[Proposition 2.3, Proposition 5.5]{MR80}.
If $M$ is isomorphic to the Riemann sphere, then $K_M$ is negative and $K^*_M=TM$ is globally generated.
Also from either Corollary \ref{cor:g.g case} or Theorem \ref{thm:RS}, we see that the multiplicity of the Grauert blow down 
of $K_M$ is 2, equal to the Euler characteristic of $M$.

The above example naturally motivates us to consider complex manifolds of higher dimensions.
Let $M$ be a compact complex manifold such that $K_M$ is ample or negative.
We can define an intrinsic invariant $\tau(M)$ of $M$ as follows: $\tau(M)$ is the multiplicity of the Grauert blow down of $K_M$ if $K_M$ is negative, 
and is the multiplicity of the Grauert blow down of $K^*_M$ if $K_M$ is ample.
So it is natural to ask which properties of $M$ can be encoded by $\tau(M)$?
If $M$ is a Riemann surface, the above discussion shows that $\tau(M)=|\chi(M)|$, 
which is the absolute value of the Euler characteristic of $M$ and hence is a topological invariant.
On the other hand, we strongly believe that $\tau(M)$ should depend on the complex structure on $M$ if $M$ is of higher dimension.


We now specialize further to the case that $M$ is a compact Riemann surface of genus $g$, and the line bundle $L=L_P$ is generated by a single point $P\in M$, viewed as a divisor of $M$. 
We first recall some notions.
\begin{defn}
An integer $k\geq 1$ is called a \h{nongap} of $M$ at $P$ if there exists a meromorphic function $h$ on $M$ such that $p$ is the unique pole of $h$ and $\mathrm{ord}_P(h)=-k$; 
we call $P$ a \h{Weierstrass point} of $M$ if the minimal  nongap of $M $ at $P$ is $g+1$.
\end{defn}

\begin{cor}\label{cor:Weierstrass point}
Let $M$ be a compact Riemann surface of genus $g$, $P\in M$. Suppose $(Z,z_0)$ is the normal Grauert blow-down of $L_{-P}$, the line bundle generated by the divisor $-P$, then
\begin{align*} 
\mathrm{mult}(Z,z_0)=\text{the\ first\ nongap\ of\ $M$\ at\ $p$}.
\end{align*}
In particular, if $P$ is not a Weierstrass point of $M$, then
\begin{align*}
\mathrm{mult}(Z,z_0)=g+1.
\end{align*}
\end{cor}

\begin{rem}
After attending two lectures in two conferences on the main results of the present work given by the first author, 
Takeo Ohsawa pointed out to the first author that the same result as in Corollary \ref{cor:Weierstrass point} was also proved by Goto and Watanabe and sent him the reference \cite[\S 5.2]{GW78};
later Takeo Ohsawa also sent to the first author two papers that discuss multiplicities of singularities of some special complex spaces \cite{Tom02}\cite{Tom24}.
Our method is based on pluripotential theory and very different from the ideas of the above references.
The authors are very grateful to Professor Ohsawa for his interest in their works and kindly providing the literatures.
\end{rem}

The remaining of the article is organized as follows. In \S \ref{sec:mult vs lelong}, we give the proof of Theorem \ref{thm:mult vs lelong}. In \S \ref{sec:lelong vs MA}, we give the proof of Theorem \ref{thm:main thm}, Corollary \ref{cor:two dim} and Theorem \ref{thm:ineq}, and in the final \S \ref{sec:blow down bundles}, we give the proof of the Theorem \ref{thm:vol control}, Theorem \ref{thm:mult of (Z,z_0)}, Theorem \ref{thm:RS} and Corollary \ref{cor:Weierstrass point}.\\


\h{Acknowledgements.}
This research is supported by National Key R\&D Program of China (No. 2021YFA1003100),
NSFC grants (No. 11871451, 12071310), and the Fundamental Research Funds for the Central Universities.

\section{Multiplicities and generalized Lelong numbers}\label{sec:mult vs lelong}
The aim of this section is to represent mixed multiplicities of analytic local algebras by generalized Lelong numbers and give the proof of Theorem \ref{thm:mult vs lelong}.
The discussion involves Lelong numbers of currents given by integration along analytic sets, for more details about this direction, we refer the readers to \cite{Dem93} (see also \cite[Chap III]{Dem}).

The following lemma is a special case of Demailly's comparison theorem of Lelong numbers \cite[Chap III Theorem 7.1, Theorem 7.8]{Dem}.
\begin{lem}\label{lem:comparison thm}
Let $(Y,y)\subseteq (\mathbb{C}^M,0)$ be a germ of analytic set of pure dimension $n$. 
Suppose $g_i=(g_{i,1},\cdots,g_{i,N_i})$, $h_j=(h_{j,1},\cdots,h_{j,N_j})$, $i,j=1,\cdots,n$ are tuples of germs of holomorphic functions in $\mathcal{O}_{\mathbb{C}^M,0}$ such that 
\begin{align*}
Y\cap\{g_{i,1}=\cdots=g_{i,N_i}=0\}=Y\cap\{h_{j,1}=\cdots=h_{j,N_j}=0\}=\{0\}.
\end{align*}
If
\begin{align*}
\limsup \frac{\log|g_j|}{\log|h_j|}(z)=l_j,\ \ \ \mathrm{when}\ z\in Y\setminus\{0\},\ z\ra0,
\end{align*}
it holds that
\begin{align*}
\nu([Y],0;\log|g_1|;\cdots;\log|g_n|)\leq l_1\cdots l_n \nu([Y],0;\log|h_1|;\cdots;\log|h_n|).
\end{align*}
\end{lem}

\begin{rem}
Given $\mathfrak{M}_{Y,y}$-primary ideals $\mathfrak{U}_j$, $j=1,\cdots,n$, Lemma \ref{lem:comparison thm} implies that 
the mixed Lelong number $\nu([Y],0;\log|g_1|;\cdots;\log|g_n|)$ is independent of the choice of the generators of $\mathfrak{U}_j$.
\end{rem}

We first prove Theorem \ref{thm:mult vs lelong} for the case that $\mathfrak{U}_1=\cdots=\mathfrak{U}_n=\mathfrak{U}$. 

\begin{prop}\label{prop:single ideal}
Let $(Y,y)$ be a germ of $n$-dimensional locally irreducible reduced complex space, $\mathfrak{M}$ be the maximal ideal in $\mathcal{O}_{Y,y}$ and $\mathfrak{U}$ be an $\mathfrak{M}$-primary ideal. 
Assuming that $(Y,y)$ is embedded in some $(\mathbb{C}^M,0)$ and that $\mathfrak{U}$ is generated by $g_1,\cdots,g_N \in\mathcal{O}_{\mathbb{C}^M,0}$, then it holds that 
$$e(\mathfrak{U})=\nu([Y],\log|g|):=\nu([Y],0;\log|g|,\cdots, \log|g|).$$ 
\end{prop}
\begin{proof}
The assumption that $\mathfrak{U}$ is $\mathfrak{M}$-primary implies that $Y\cap\{g_1=\cdots=g_N=0\}=\{0\}$. 
Therefore, the holomorphic map 
\begin{align*}
g:=(g_1,\cdots,g_N):(\mc^M_w,0)\rightarrow (\mc^N_z,0)
\end{align*}
induces a finite branched covering between varieties $Y$ and $Z:=g(Y)$ (see, e.g. \cite[Theorem 7, E]{Gun90}). 
Let $d$ be the degree of this covering, then we have $g_*[Y]=d[Z]$ in the sense of currents near $0\in\mc^N$. 
Using the transformation fomula of Lelong numbers under direct image \cite[Proposition 9.3, Chap III]{Dem}, we obtain that
\begin{align*}
\nu([Y],\log|g|)=\nu(g_*[Y],\log|z|)=d\nu([Z],\log|z|).
\end{align*}
By a famous theorem of Thie \cite{Thi67}, $\nu([Z],\log|z|)$ is the geometric multiplicity of the variety $Z$ at $0$. Therefore, up to a generic linear transform of $\mathbb{C}^N$, the projection
\begin{align*}
\pi:(\mathbb{C}^N,0)&\rightarrow (\mathbb{C}^n,0)\\
(z_1,\cdots,z_N)&\mapsto(z_1,\cdots,z_n)
\end{align*}
induces a branched covering $\pi:(Z,0)\rightarrow(\mathbb{C}^n,0)$ of degree $\nu([Z],\log|z|)$. 
Moreover, for every $k=n+1,\cdots,N$, there exists a Weierstrass polynomial
\begin{align*}
z_k^{s_k}-\sum_{j=1,\cdots, s_k}a_{j,k}(z_1,\cdots,z_{k-1})z_k^{s_k-j}=0\in\mathcal{O}_{Z,0} \tag{*}
\end{align*}
with coefficients $a_{j,k}(z_1,\cdots,z_{k-1}) \in \mathfrak{M}_{\mathbb{C}^{k-1},0}^j$ (see e.g. \cite[Proposition 4.8, Chap II]{Dem}).

Note that $\pi\circ g:Y\rightarrow\mathbb{C}^n$ is a finite branched covering of degree 
$$d\nu([Z],\log|z|)=\nu([Y],\log|g|).$$ 
By Samuel's fomula (see e.g. \cite[Proposition A.10]{Mum76}), the degree of $\pi\circ g$ coincides with $e\big((g_1,\cdots,g_n)\big)$, that is
\begin{align*}
\nu([Y],\log|g|)=e\big((g_1,\cdots,g_n)\big).
\end{align*}
We will complete the proof by showing that 
\begin{align*}
e\big((g_1,\cdots,g_n)\big)=e\big((g_1,\cdots,g_{n+1})\big)=\cdots=e\big((g_1,\cdots,g_N)\big)
\end{align*}
inductively. 
Indeed, $(*)$ implies that, for $k=n+1,\cdots,N$,
\begin{align*}
g_{k}^{s_{k}}-\sum_{j=1,\cdots, s_{k}}a_{j,n+1}(g_1,\cdots,g_{k-1})g_{k}^{s_{k}-j}=0\in\mathcal{O}_{Y,0}, \tag{**}
\end{align*}
where $a_{j,k}(g_1,\cdots,g_{k-1})\in (g_1,\cdots,g_{k-1})^j$.
Hence $g_{k}$ belongs to the integral closure of $(g_1,\cdots,g_{k-1})$, 
 a theorem of D. Rees \cite{Re61} yields that the multiplicity $e((g_1,\cdots,g_k))=e((g_1,\cdots,g_{k-1}))$. 

For the sake of the completeness, we provide
 a direct proof of the equality here. Set
\begin{align*}
P_n(k)&:=\dim_{\mathbb{C}}\big(\mathcal{O}_{Y,y}/(g_1,\cdots,g_n)^k \big),\\
P_{n+1}(k)&:=\dim_{\mathbb{C}}\big(\mathcal{O}_{Y,y}/(g_1,\cdots,g_{n+1})^k\big).
\end{align*}
It is clear that $P_{n}(k)\geq P_{n+1}(k)$, we are going to prove $P_{n}(k)\leq P_{n+1}(k+s_{n+1})$ by showing that
\begin{align*}
(g_1,\cdots,g_{n+1})^{k+s_{n+1}}\subseteq(g_1,\cdots,g_n)^k.
\end{align*}
If so, the desired equality $e\big((g_1,\cdots,g_n)\big)=e\big((g_1,\cdots,g_{n+1})\big)$ follows from the definition directly.

We take $g_1^{\alpha_1}\cdot...\cdot g_{n+1}^{\alpha_{n+1}}$, $ \alpha_1+\cdots+\alpha_{n+1}=k+s_{n+1}$. 
If $\alpha_{n+1}\leq s_{n+1}$, then $\alpha_1+\cdots+\alpha_n\geq k$, there is nothing to prove. If $\alpha_{n+1}>s_{n+1}$, we use $(**)$ for $k=n+1$:
\begin{align*}
g_{n+1}^{s_{n+1}}-\sum_{j=1,\cdots, s_{n+1}}a_{j,n+1}(g_1,\cdots,g_n)g_{n+1}^{s_{n+1}-j}=0,
\end{align*}
and $a_{j,n+1}(g_1,\cdots,g_n)=\sum_{\beta_1+\cdots+\beta_n=j}F_{j,\beta_1,\cdots,\beta_n}(g_1,\cdots,g_n)g_1^{\beta_1}\cdot...\cdot g_n^{\beta_n}$, 
which leads to
\begin{align*}
g_1^{\alpha_1}\cdot...\cdot g_{n+1}^{\alpha_{n+1}}=\sum^{s_{n+1}}_{j=1}\sum_{\beta_1+\cdots+\beta_n=j} F_{j,\beta_1,\cdots,\beta_n}(g_1,\cdots,g_n)g_1^{\alpha_1+\beta_1}\cdot...\cdot g_n^{\alpha_n+\beta_n}g_{n+1}^{\alpha_{n+1}-j}.
\end{align*}
If the degree of $g_{n+1}$ in some term is still greater than $s_{n+1}$, we just repeat the arguments as above. 
We finally get that
\begin{align*}
g_1^{\alpha_1}\cdot...\cdot g_{n+1}^{\alpha_{n+1}}=\sum G_{\gamma_1,\cdots,\gamma_n}(g_1,\cdots,g_n) g_{1}^{\gamma_1}\cdot...\cdot g_n^{\gamma_n},\gamma_1+\cdots+\gamma_n\geq k.
\end{align*}
This proves $(g_1,\cdots,g_{n+1})^{k+s_{n+1}}\subseteq(g_1,\cdots,g_n)^k$ as desired.
\end{proof}

We are going to prove Theorem \ref{thm:mult vs lelong}.

\begin{thm}(=Theorem \ref{thm:mult vs lelong})
Let $(Y,y)\subset(\mc^M,0)$ be a germ of $n$-dimensional irreducible reduced analytic subset, 
$\mathfrak{M}$ be the maximal ideal in $\mathcal{O}_{Y,y}$, and $\mathfrak{U}_1,\cdots,\mathfrak{U}_n \subset \mathcal{O}_{Y,y}$ 
be $n$ $\mathfrak{M}$-primary ideals.
Assuming that $\mathfrak{U}_j$ is generated by $g_{j,1},\cdots,g_{j,N_j} \in\mathcal{O}_{\mc^M,0}$, then it holds that 
\begin{align*}
e(\mathfrak{U}_1;\cdots;\mathfrak{U}_{n})=\nu([Y],0;\log|g_1|,\cdots,\log|g_n|).
\end{align*}
\end{thm}
The proof of the theorem makes use of the following easily checked polarization property of mixed multiplicity: for $p_1,\cdots,p_n\in\mathbb{N}^*$,
\begin{align*}
e(\mathfrak{U}_1^{p_1}\cdot...\cdot\mathfrak{U}_n^{p_n})=\sum_{d_1+\cdots+d_n=n}\frac{n!}{d_1!\cdot...\cdot d_n!}e(\mathfrak{U}_1^{[d_1]};\cdots;\mathfrak{U}_n^{[d_n]})p_1^{d_1}\cdot...\cdot p_n^{d_n}.
\end{align*}

\begin{proof}
Since $\big\{\prod_{j=1}^n (g_{j,\alpha^j_1}\cdot...\cdot g_{j,\alpha^j_{p_j}}):1\leq \alpha^j_{1},\cdots,\alpha^j_{p_j}\leq N_j\big\}$ generates the ideal $\mathfrak{U}_1^{p_1}\cdot...\cdot\mathfrak{U}_n^{p_n}$, Proposition \ref{prop:single ideal} then yields that
\begin{align*}
e(\mathfrak{U}_1^{p_1}\cdot...\cdot\mathfrak{U}_n^{p_n})=\lim_{r\rightarrow0^+}\int_{\{|z|<r\}}[Y]\wedge\Big(\frac{1}{2}dd^c\log(\sum_{\alpha} \big|\prod_{j=1}^n g_{j,\alpha^j_1}\cdot...\cdot g_{j,\alpha^j_{p_j}} \big|^2)\Big)^n.
\end{align*}
Easy computations lead to the following
\begin{align*}
&\int_{\{|z|<r\}}[Y]\wedge\Big(\frac{1}{2}dd^c\log(\sum_{\alpha}\big|\prod_{j=1}^n g_{j,\alpha^j_1}\cdot...\cdot g_{j,\alpha^j_{p_j}}\big|^2)\Big)^n\\
=&\int_{\{|z|<r\}}[Y]\wedge\Big(\sum_{j=1}^n \frac{1}{2}dd^c\log\big(|g_{j,1}|^2+\cdots+|g_{j,N_j}|^2\big)^{p_j}\Big)^n\\
=&\sum_{d_1+\cdots+d_n=n}\frac{n!}{d_1!\cdot...\cdot d_n!}\Big(\int_{\{|z|<r\}}[Y]\wedge(dd^c\log|g_1|)^{d_1}\wedge\cdots\wedge(dd^c\log|g_n|)^{d_n}\Big)p_1^{d_1}\cdot...\cdot p_n^{d_n}.
\end{align*}
Since $p_1,\cdots,p_n$ are arbitrary, we can infer that
\begin{align*}
e(\mathfrak{U}_1^{[d_1]};\cdots;\mathfrak{U}_n^{[d_n]})=\lim_{r\rightarrow0^+}\int_{\{|z|<r\}}[Y]\wedge(dd^c\log|g_1|)^{d_1}\wedge\cdots\wedge(dd^c\log|g_n|)^{d_n}.
\end{align*}
In particular, letting $d_1=\cdots=d_n=1$, we complete the proof.
\end{proof}

\begin{rem}
The facts that $e(\mathfrak{U}_1;\cdots;\mathfrak{U}_n)>0$ and the inequality $e(\mathfrak{U}_1;\cdots;\mathfrak{U}_n)\geq e(\mathfrak{V}_1;\cdots;\mathfrak{V}_n)$ provided $\mathfrak{U}_j\subseteq\mathfrak{V}_j$, $j=1,\cdots,n$, 
are well-known. However they are not obvious in view of the definition. 
On the other hand, they follow from Theorem \ref{thm:mult vs lelong} and Lemma \ref{lem:comparison thm} directly.
\end{rem}

\section{Generalized Lelong numbers, intersection numbers and Monge-Amp\`ere products}\label{sec:lelong vs MA}

We prove that the generalized Lelong numbers can be transfromed into intersection numbers and Monge-Amp\`ere products by proper modifications and complete the proof of Theorem \ref{thm:main thm}.

\begin{thm}(=Theorem \ref{thm:main thm})\label{thm:lelong number vs intersection number}
Assume that the germ $(Y,y)$ is locally irreducible and reduced, then it holds that
\begin{align*}
e\big(\mathfrak{U}_1;\cdots;\mathfrak{U}_n\big)=(-1)^{n-1}(D_1)\cdot...\cdot(D_n)+\int_ET_1 \wedge \cdots\wedge T_n,
\end{align*}
provided that $\{A_j\}$ satisfies the dimension conditions
$$\mathrm{codim}(A_{j_1}\cap\cdots\cap A_{j_m})\geq m$$ 
for all choices of indices $j_1<\cdots<j_m$ in $\{1,\cdots,n\}$.
\end{thm}

The main ingredients of our proof are the approximation techniques in \cite[Theorem A]{Pan24} (see also \cite[Proposition 4.2]{PT23}) and the generalized Monge-Amp\`ere products introduced by Demailly.  
\begin{thm}\cite[Theorem 2.5]{Dem93}
Let $u_1,\cdots,u_q$ $(q\leq n)$ be plurisubharmonic functions on an open subset $\Omega\subseteq\mathbb{C}^n$ with analytic singularities $A_1,\cdots,A_q$. 
Then $u_1dd^cu_2\wedge\cdots\wedge dd^cu_q$ and $dd^cu_1\wedge dd^cu_2\wedge\cdots\wedge dd^cu_q$ are well-defined as long as
\begin{align*}
\mathrm{codim}(A_{j_1}\cap \cdots\cap A_{j_m})\geq m
\end{align*} 
for all choices of indices $j_1<\cdots<j_m$ in $\{1,\cdots,q\}$.
\end{thm}
\begin{proof}[Proof of Theorem \ref{thm:lelong number vs intersection number}]
Assuming that $(Y,y)$ is embedded in some $(\mc_z^M,0)$, and $g_{j,k}\in\mathcal{O}(\{|z|<r_0\})$ for some $r_0>0$. 
We decompose the current $dd^c\pi^*\log|g_j|$ in $X':=\pi^{-1}(\{|z|<r_0\})$ as follows:
\begin{align*}
dd^c\pi^*\log|g_j|&=\sum_{i=1}^N m_{j,i}[E_i]+T_j,\\
T_j&=\theta_j+dd^c\varphi_j,
\end{align*}
where
\begin{itemize}
\item
$\theta_j$ is a smooth closed $(1,1)$-form, with $\{\theta_j\}=\{-\sum_{i=1}^Nm_{j,i}[E_i]\}\in H^{(1,1)}(X',\mathbb{Z})$,
\item
$\varphi_j\in\mathrm{C}^{\infty}(X'\setminus E)\cap \mbox{PSH}(X',\theta_j)$ has analytic singularities $A_j$, $\mathrm{codim}(A_j)\geq2$.
\end{itemize}
Indeed, let $L_j:=\mathcal{O}_X(D_j)$ be the line bundle on $X$ associated to the divisor $D_j=\sum_{i=1}^N m_{j,i}E_i$.
Let $s_j$ be the canonical section of $L_j$ such that $\mathrm{div}(s_j)=D_j$. 
We take an arbitrary smooth metric $h_j=e^{-\psi_j}$ on $L_j$. Then we have
\begin{align*}
dd^c\log|s_j|_{h_j}=\sum_{i=1}^Nm_{j,i}[E_i]-dd^c\psi_j.
\end{align*}
Therefore, $\theta_j=-dd^c\psi_j$ and $\varphi_j=\pi^*\log|g_j|-\log|s_j|_{h_j}$ are globally defined and satisfy our requirements.

For $0<r<<1$, we take the cut-off function $\tilde{\chi}_r$ such that $\tilde{\chi}_r\equiv1$ in $\{|z|<\frac{r}{2}\}$ and $\mathrm{supp}(\tilde{\chi}_r)\subseteq\{|z|<r\}$, and set $\chi_r:=\pi^*\tilde{\chi}_r$.
Then the mixed Lelong number is expressed as follows:
\begin{align*}
&\nu([Y],0;\log|g_1|,\cdots,\log|g_n|)\\
=&\lim_{r\rightarrow0^+}\lim_{k_1\rightarrow\infty}\cdots\lim_{k_n\rightarrow\infty}\int_{Y}\tilde{\chi}_r(dd^c\log(|g_1|+\frac{1}{k_1}))\wedge\cdots\wedge(dd^c\log(|g_n|+\frac{1}{k_n})).
\end{align*}
First, we let $k_n\ra\infty$ and get
\begin{align*}
&\lim_{k_n\rightarrow\infty}\int_{Y}\tilde{\chi}_r\bigwedge_{j=1}^n dd^c\log(|g_j|+\frac{1}{k_j}) \\
=&\lim_{k_n\rightarrow\infty}\int_{X}\chi_r\bigwedge_{j=1}^n dd^c\pi^*\log(|g_j|+\frac{1}{k_j})\\
=&\int_{X}\chi_r\bigwedge_{j=1}^{n-1}dd^c\pi^*\log(|g_j|+\frac{1}{k_j})\wedge(\sum_{i=1}^Nm_{n,i}[E_i]+T_n).
\end{align*}
Since $\sum_{i=1}^Nm_{n,i}[E_i]$ is closed with compact support $E$ and in some neighbourhood of $E$, 
\begin{align*}
\chi_r\bigwedge_{j=1}^{n-1}dd^c\pi^*\log(|g_j|+\frac{1}{k_j})=dd^c\Big(\chi_r\pi^*\log(|g_1|+\frac{1}{k_1})\wedge\bigwedge_{j=2}^{n-1}dd^c\pi^*\log(|g_j|+\frac{1}{k_j})\Big)
\end{align*}
is a smooth $dd^c$-exact form, we can infer that
\begin{align*}
\int_{X}\chi_r\bigwedge_{j=1}^{n-1}dd^c\pi^*\log(|g_j|+\frac{1}{k_j})\wedge\big(\sum_{i=1}^Nm_{n,i}[E_i]\big)=0.
\end{align*}
Therefore,
\begin{align*}
\lim_{k_n\rightarrow\infty}&\int_X\chi_r\bigwedge_{j=1}^{n}dd^c\pi^*\log(|g_j|+\frac{1}{k_j})=\int_X \chi_r\bigwedge_{j=1}^{n-1}dd^c\pi^*\log(|g_j|+\frac{1}{k_j})\wedge T_n.
\end{align*}
We hope to repeat the above arguments and get that
\begin{align*}
&\lim_{r\ra 0^+}\cdots\lim_{k_{n-1}\rightarrow\infty} \int_{X}\chi_r\bigwedge_{j=1}^n dd^c\pi^*\log(|g_j|+\frac{1}{k_j})\wedge T_n\\
=&\lim_{r\ra0^+}\cdots \lim_{k_{n-2}\ra \infty} \int_{X}\chi_r\bigwedge_{j=1}^{n-1} dd^c\pi^*\log(|g_j|+\frac{1}{k_j}) \wedge T_{n-1}\wedge T_n\\
=&\cdots\\
=&\lim_{r\ra0^+}\cdots \lim_{k_{n-m}\ra\infty} \int_{X}\chi_r\bigwedge_{j=1}^{n-m} dd^c\pi^*\log(|g_j|+\frac{1}{k_j})\wedge T_{n-m+1}\wedge\cdots\wedge T_n\\
=&\cdots
\end{align*}
However, when letting $k_{n-m}\ra \infty$, 
we need $\mathrm{codim}(E\cap A_{n-m+1}\cap\cdots A_n)\geq m$ to make the current $dd^c\pi^*\log|g_{n-m-1}|\wedge T_{n-m}\wedge\cdots\wedge T_n$ well-defined. 
This is stronger than our hypothesis, 
which says that $\mathrm{codim}(E\cap A_{n-m+1}\cap\cdots A_n)=\mathrm{codim}( A_{n-m+1}\cap\cdots A_n)\geq m-1$, 
so we have to modify our proof.

We split $T_n$ into $\theta_n+dd^c\varphi_n$ and use integration by parts to get that
\begin{align*}
&\int_X \chi_r\bigwedge_{j=1}^{n-1}dd^c\pi^*\log(|g_j|+\frac{1}{k_j})\wedge T_n\\
=&\int_X\chi_r \bigwedge_{j=1}^{n-1}dd^c\pi^*\log(|g_j|+\frac{1}{k_j})\wedge \theta_n
+\int_X\bigwedge_{j=1}^{n-1}dd^c\pi^*\log(|g_j|+\frac{1}{k_j})\wedge\varphi_ndd^c\chi_r.
\end{align*}
Note that $\varphi_ndd^c\chi_r$ is no longer singular but a smooth $(1,1)$-form. 
When letting $k_{n-1}\rightarrow\infty$, it holds that
\begin{align*}
\lim_{k_{n-1}\rightarrow\infty}\lim_{k_n\ra\infty}&\int_X\chi_r \bigwedge_{j=1}^{n}dd^c\pi^*\log(|g_j|+\frac{1}{k_j})\\
=&\int_X\chi_r \bigwedge_{j=1}^{n-2}dd^c\pi^*\log(|g_j|+\frac{1}{k_j})\wedge(\sum_{i=1}^Nm_{n-1,i}[E_i]+\theta_{n-1}+dd^c\varphi_{n-1})\wedge \theta_n\\
+&\int_X \bigwedge_{j=1}^{n-2}dd^c\pi^*\log(|g_j|+\frac{1}{k_j})\wedge(\sum_{i=1}^Nm_{n-1,i}[E_i]+\theta_{n-1}+dd^c\varphi_{n-1})\wedge \varphi_ndd^c\chi_r.
\end{align*}
As before, $\chi_r \bigwedge_{j=1}^{n-2}dd^c\pi^*\log(|g_j|+\frac{1}{k_j})\wedge \theta_n$ is exact near $E$, 
so $\sum_{i=1}^Nm_{n-1,i}[E_i]$ does not contribute to the integration in the first term.
Moreover, $\varphi_ndd^c\chi_r$ vanishes near $E$, thus $\sum_{i=1}^Nm_{n-1,i}[E_i]$ does not contribute to the integration in the second term neither. 
We also use the following integration by parts
\begin{align*}
&\int_X\chi_r \bigwedge_{j=1}^{n-2}dd^c\pi^*\log(|g_j|+\frac{1}{k_j})\wedge dd^c\varphi_{n-1} \wedge \theta_n=\int_X \bigwedge_{j=1}^{n-2}dd^c\pi^*\log(|g_j|+\frac{1}{k_j})\wedge \varphi_{n-1}dd^c\chi_r \wedge \theta_n,\\
&\int_X \bigwedge_{j=1}^{n-2}dd^c\pi^*\log(|g_j|+\frac{1}{k_j})\wedge dd^c\varphi_{n-1} \wedge \varphi_ndd^c\chi_r=\int_X \bigwedge_{j=1}^{n-2}dd^c\pi^*\log(|g_j|+\frac{1}{k_j})\wedge \varphi_{n-1}dd^c\chi_r \wedge dd^c\varphi_n.
\end{align*}
Therefore, we get that
\begin{align*}
\lim_{k_{n-1}\rightarrow\infty}\lim_{k_n\ra\infty} 
&\int_X\chi_r \bigwedge_{j=1}^{n}dd^c\pi^*\log(|g_j|+\frac{1}{k_j}) \\
=&\int_X\chi_r \bigwedge_{j=1}^{n-2}dd^c\pi^*\log(|g_j|+\frac{1}{k_j})\wedge \theta_{n-1}\wedge \theta_n\\
+&\int_X  \bigwedge_{j=1}^{n-2}dd^c\pi^*\log(|g_j|+\frac{1}{k_j})\wedge\varphi_{n-1}dd^c\chi_r\wedge\theta_n
\end{align*}
\begin{align*}
+&\int_X  \bigwedge_{j=1}^{n-2}dd^c\pi^*\log(|g_j|+\frac{1}{k_j})\wedge \theta_{n-1} \wedge \varphi_ndd^c\chi_r\\
+&\int_X  \bigwedge_{j=1}^{n-2}dd^c\pi^*\log(|g_j|+\frac{1}{k_j})\wedge \varphi_{n-1}dd^c\chi_r\wedge dd^c\varphi_n.
\end{align*}
Clearly, all terms in the integration are still smooth. 
They can be viewed as the ``smooth" expansion of 
$\int_X\chi_r \bigwedge_{j=1}^{n-2}(\theta_{n-1}+dd^c\varphi_{n-1})\wedge (\theta_n+dd^c\varphi_n)$. 
Letting $k_{n-2}\ra\infty$, similar arguments yields that
\begin{align*}
\lim_{k_{n-2}\ra\infty}\lim_{k_{n-1}\rightarrow\infty}\lim_{k_n\ra\infty} 
&\int_X\chi_r \bigwedge_{j=1}^{n}dd^c\pi^*\log(|g_j|+\frac{1}{k_j}) \\
=&\int_X\chi_r \bigwedge_{j=1}^{n-3}dd^c\pi^*\log(|g_j|+\frac{1}{k_j})\wedge \theta_{n-2}\wedge \theta_{n-1}\wedge \theta_n\\
+&\int_X \bigwedge_{j=1}^{n-3}dd^c\pi^*\log(|g_j|+\frac{1}{k_j})\wedge \varphi_{n-2}dd^c\chi_r\wedge \theta_{n-1}\wedge \theta_n\\
+&\int_X \bigwedge_{j=1}^{n-3}dd^c\pi^*\log(|g_j|+\frac{1}{k_j})\wedge \theta_{n-2}\wedge \varphi_{n-1} dd^c\chi_r \wedge \theta_n\\
+&\int_X \bigwedge_{j=1}^{n-3}dd^c\pi^*\log(|g_j|+\frac{1}{k_j})\wedge \varphi_{n-2}dd^c\chi_r\wedge dd^c\varphi_{n-1} \wedge \theta_n\\
+&\int_X \bigwedge_{j=1}^{n-3}dd^c\pi^*\log(|g_j|+\frac{1}{k_j})\wedge \theta_{n-2}\wedge \theta_{n-1} \wedge \varphi_ndd^c\chi_r\\
+&\int_X \bigwedge_{j=1}^{n-3}dd^c\pi^*\log(|g_j|+\frac{1}{k_j})\wedge \varphi_{n-2}dd^c\chi_r\wedge \theta_{n-1} \wedge dd^c\varphi_n\\
+&\int_X \bigwedge_{j=1}^{n-3}dd^c\pi^*\log(|g_j|+\frac{1}{k_j})\wedge \theta_{n-2}\wedge \varphi_{n-1}dd^c\chi_r \wedge dd^c\varphi_n\\
+&\int_X \bigwedge_{j=1}^{n-3}dd^c\pi^*\log(|g_j|+\frac{1}{k_j})\wedge \varphi_{n-2}dd^c\chi_r\wedge dd^c\varphi_{n-1} \wedge dd^c\varphi_n.
\end{align*}
The terms in the integration are still smooth, they can be viewed as the ``smooth" expansion of
\begin{align*}
\int_X\chi_r\bigwedge_{j=1}^{n-3}dd^c\pi^*\log(|g_j|+\frac{1}{j})\wedge (\theta_{n-2}+dd^c\varphi_{n-2})\wedge(\theta_{n-1}+dd^c\varphi_{n-1})\wedge(\theta_n+dd^c\varphi_n).
\end{align*} 
Repeat the arguments and we eventually get that
\begin{align*}
\lim_{k_1\rightarrow\infty}\cdots\lim_{k_n\rightarrow\infty}&\int_X\chi_r\bigwedge_{j=1}^n dd^c\pi^*\log(|g_j|+\frac{1}{k_j})\\
=&\int_X\chi_r(\sum_{i=1}^Nm_{1,i}[E_i]+\theta_1+dd^c\varphi_1)\wedge\theta_2\wedge\theta_3\wedge\cdots\wedge\theta_n\\
+&\int_X(\sum_{i=1}^Nm_{1,i}[E_i]+\theta_1+dd^c\varphi_1)\wedge \varphi_2dd^c\chi_r\wedge \theta_3\wedge\cdots\wedge \theta_n\\
+&\cdots\cdots\\
+&\cdots\cdots\\
+&\int_X(\sum_{i=1}^Nm_{1,i}[E_i]+\theta_1+dd^c\varphi_1)\wedge\theta_2\wedge\varphi_3dd^c\chi_r\wedge dd^c\varphi_4\wedge\cdots\wedge dd^c\varphi_n\\
+&\int_X(\sum_{i=1}^Nm_{1,i}[E_i]+\theta_1+dd^c\varphi_1)\wedge\varphi_2dd^c\chi_r\wedge dd^c\varphi_3\wedge dd^c\varphi_4\wedge\cdots\wedge dd^c\varphi_n.
\end{align*}
The first term contributes to the intersection numbers:
\begin{align*}
\int_X\chi_r\sum_{i=1}^Nm_{1,i}[E_i]\wedge\theta_2\wedge\cdots\wedge\theta_n=D_1\cdot(-D_2)\cdot...\cdot(-D_n)=(-1)^{n-1}(D_1)\cdot...\cdot (D_n).
\end{align*}
The dimension conditions ensure the Monge-Amp\`ere products $dd^c\varphi_{j_1}\wedge\cdots\wedge dd^c\varphi_{j_m}$ is well-defined for all choice of indices $j_1<\cdots<j_m$, therefore we have
\begin{align*}
&\int_X\chi_r\theta_1\wedge\theta_2\wedge\theta_3\wedge\cdots\wedge \theta_n
+\int_X\varphi_1dd^c\chi_r\wedge\theta_2\wedge\theta_3\wedge\cdots\wedge\theta_n\\
+&\cdots\cdots\\
+&\cdots\cdots\\
+&\int_X\theta_1\wedge\varphi_2dd^c\chi_r\wedge dd^c\varphi_3\wedge\cdots\wedge dd^c\varphi_n+\int_X\varphi_1dd^c\chi_r\wedge dd^c\varphi_2\wedge dd^c\varphi_3\wedge\cdots\wedge dd^c\varphi_n\\
=&\int_X\chi_r\theta_1\wedge\theta_2\wedge\theta_3\wedge\cdots\wedge \theta_n
+\int_X\chi_rdd^c\varphi_1\wedge\theta_2\wedge\theta_3\wedge\cdots\wedge\theta_n\\
+&\cdots\cdots\\
+&\cdots\cdots\\
+&\int_X\chi_r\theta_1\wedge dd^c\varphi_2\wedge dd^c\varphi_3\wedge\cdots\wedge dd^c\varphi_n
+\int_X\chi_rdd^c\varphi_1\wedge dd^c\varphi_2\wedge dd^c\varphi_3\wedge\cdots\wedge dd^c\varphi_n\\
=&\int_X\chi_r(\theta_1+dd^c\varphi_1)\wedge\cdots\wedge(\theta_n+dd^c\varphi_n).
\end{align*}
Letting $r\rightarrow0^+$, we obtain the desired formula
\begin{align*}
\nu([Y],0;\log|g_1|;\cdots;\log|g_n|)=(-1)^{n-1}(D_1)\cdot...\cdot (D_n)+\int_ET_1\wedge\cdots\wedge T_n.
\end{align*}

\end{proof}
\begin{rem}
In our proof, the equality 
\begin{align*}
 &\ \nu([Y],0;\log|g_1|;\cdots;\log|g_n|)-(-1)^{n-1}(D_1)\cdot...\cdot( D_n)\\
=&\lim_{r\rightarrow0^+}
\Big(\int_X\chi_r\theta_1\wedge\theta_2\wedge\theta_3\wedge\cdots\wedge \theta_n+\int_X\varphi_1dd^c\chi_r\wedge\theta_2\wedge\theta_3\wedge\cdots\wedge\theta_n\\
+&\cdots\cdots\\
+&\cdots\cdots\\
+&\int_X\theta_1\wedge\varphi_2dd^c\chi_r\wedge dd^c\varphi_3\wedge\cdots\wedge dd^c\varphi_n+\int_X\varphi_1dd^c\chi_r\wedge dd^c\varphi_2\wedge dd^c\varphi_3\wedge\cdots\wedge dd^c\varphi_n\Big).
\end{align*}
does not require the dimension conditions.
\end{rem}

If we drop the dimension conditions on $A_j$, we can obtain an inequality provided that $E$ is irreducible and $\mathcal{O}_{X}(-E)$ is semi-positive near $E$.
\begin{cor}(=Theorem \ref{thm:ineq})
Assume that the germ $(Y,y)$ is locally irreducible and reduced and $E$ is irreducible, if $\mathcal{O}_X(-E)$ is semi-positive on some neighborhood of $E$ in $X$, 
then it holds that
\begin{align*}
e\big(\mathfrak{U}_1;\cdots;\mathfrak{U}_n\big)\geq(-1)^{n-1}(D_1)\cdot...\cdot(D_n)\geq 0.
\end{align*}
\end{cor}
\begin{proof}
In a neighbourhood of $E$, let $e^{-\psi}$ be a smooth semi-positive metric on $\mathcal{O}_X(-E)$, then $e^{-m_j\psi}$ is a smooth semi-positive metric on $\mathcal{O}_X(-D_j)$. 
We obtain the decomposition as in Theorem \ref{thm:lelong number vs intersection number}:
\begin{align*}
dd^c\pi^*\log|g_j|=m_j[E]+T_j,
\end{align*}
with $T_j=\theta_j+dd^c\varphi_j$, such that
\begin{itemize}
\item
$\theta_j=m_jdd^c\psi$ is a closed, semi-positive $(1,1)$-form, 
\item
$\varphi_j\in\mathrm{C}^{\infty}(X'\setminus E)\cap \mbox{PSH}(X',\theta_j)$ has analytic singularities.
\end{itemize}
The same computation yields that
\begin{align*}
 &\ \nu([Y],0;\log|g_1|,\cdots,\log|g_n|)-(-1)^{n-1}(D_1)\cdot...\cdot( D_n)\\
=&\lim_{r\rightarrow0^+}
\Big(\int_X\chi_r\theta_1\wedge\theta_2\wedge\theta_3\wedge\cdots\wedge \theta_n+\int_X\varphi_1dd^c\chi_r\wedge\theta_2\wedge\theta_3\wedge\cdots\wedge\theta_n\\
+&\cdots\cdots\\
+&\cdots\cdots\\
+&\int_X\theta_1\wedge\varphi_2dd^c\chi_r\wedge dd^c\varphi_3\wedge\cdots\wedge dd^c\varphi_n+\int_X\varphi_1dd^c\chi_r\wedge dd^c\varphi_2\wedge dd^c\varphi_3\wedge\cdots\wedge dd^c\varphi_n\Big).
\end{align*}
We fix $r$ and choose $N$ sufficiently large such that $\varphi_j^N:=\max\{\varphi_j,-N\}$ coincides with $\varphi_j$ near $\{dd^c\chi_r\neq0\}$. 
Note that $\varphi_j^N$ is still $\theta_j$-psh thanks to $\theta_j$ is semi-positive. 
It follows that
\begin{align*}
&\int_X\chi_r\theta_1\wedge\theta_2\wedge\theta_3\wedge\cdots\wedge \theta_n+\int_X\varphi_1dd^c\chi_r\wedge\theta_2\wedge\theta_3\wedge\cdots\wedge\theta_n\\
+&\cdots\cdots\\
+&\cdots\cdots\\
+&\int_X\theta_1\wedge\varphi_2dd^c\chi_r\wedge dd^c\varphi_3\wedge\cdots\wedge dd^c\varphi_n+\int_X\varphi_1dd^c\chi_r\wedge dd^c\varphi_2\wedge dd^c\varphi_3\wedge\cdots\wedge dd^c\varphi_n\\
=&\int_X\chi_r\theta_1\wedge\theta_2\wedge\theta_3\wedge\cdots\wedge \theta_n+\int_X\varphi_1^Ndd^c\chi_r\wedge\theta_2\wedge\theta_3\wedge\cdots\wedge\theta_n\\
+&\cdots\cdots\\
+&\cdots\cdots\\
+&\int_X\theta_1\wedge\varphi_2^Ndd^c\chi_r\wedge dd^c\varphi_3^N\wedge\cdots\wedge dd^c\varphi_n^N+\int_X\varphi_1^Ndd^c\chi_r\wedge dd^c\varphi_2^N\wedge dd^c\varphi_3^N\wedge\cdots\wedge dd^c\varphi_n^N\\
=&\int_X\chi_r(\theta_1+dd^c\varphi_1^N)\wedge\cdots\wedge(\theta_n+dd^c\varphi_n^N)\geq0.
\end{align*}
This leads to the desired inequality
\begin{align*}
\nu([Y],0;\log|g_1|,\cdots,\log|g_n|)\geq(-1)^{n-1}(D_1)\cdot...\cdot(D_n).
\end{align*}
\end{proof}

\section{Multiplicities of the Grauert blow down of negative vector bundles}\label{sec:blow down bundles}

In this section, we study the normal blow down of negative vector bundles and give the proof of Theorem \ref{thm:vol control}, Theorem \ref{thm:mult of (Z,z_0)}, Theorem \ref{thm:RS} and Corollary \ref{cor:Weierstrass point}.

Let $L\ra M$ be an ample line bundle over compact complex manifold $M$. 
We will use the notations in Introduction (ahead of Theorem \ref{thm:vol control}).
For $f\in H^0(M, L^k)$, we can naturally viewed $f$ as a holomorphic function on $L^*$, whose restriction to each fiber
of $L^*$ is a homogeneous polynomial of degree $k$.
On the other hand, any holomorphic function $f$ on $L^*$ can be represented as a convergent series 
$$f=\sum^\infty_{k=0}f_k,\ f_k\in \in H^0(M, L^k).$$

%

\begin{prop}\label{prop:L globally generated}
Let $f:(L^*,M)\ra(Z,z_0)$ be the normal Grauert blow-down of $L^*$. Then the vanishing order of $f^*\mathfrak{M}_{Z,z_0}$ along $M$ is $k_0$ and the following statements are equivalent:
\begin{itemize}
\item[$(1)$]
the ideal subsheaf of $\mathcal O_{L^*}$ generated by $f^*\mathfrak{M}_{Z,z_0}$ is invertible,
\item[$(2)$]
$L^{k_0}$ is globally generated.
\end{itemize}

\end{prop}
\begin{proof}
We may assume that the germ $(Z,z_0)$ is embedded in $(\mathbb{C}^N_z,0)$ and set $f_j:=f^*z_j$, $j=1,\cdots,N$.
We expand $f_j$ into Taylor series on $L^*$
\begin{align*}
f_j=\sum_{k=0}f_{j,k},\ \ f_{j,k}\in H^0(M,L^{k}).
\end{align*}
It is clear $\mathrm{ord}_M(f_j)\geq k_0$ by the assumption that $H^0(M,L)=\cdots=H^0(M,L^{k_0-1})=0$.
Also by assuption, there is some $ g\in H^0(M,L^{k_0})$, $g\neq0$. 
As above, $g$ can be viewed as a holomorphic function on $L^*$. The normality of $Z$ ensures that $(f^{-1})^*g$, as a holomorphic function on $Z\backslash\{z_0\}$, is holomorphic on $Z$, so $(f^{-1})^*g$ can be expressed as power series $\sum c_{\alpha}z^{\alpha}$ near $z_0$. 
It follows that, near the zero section $M$ of $L^*$, 
\begin{align*}
g=\sum_{\alpha}c_{\alpha}f_1^{\alpha_1}\cdot...\cdot f_N^{\alpha_N}.
\end{align*}
Since $\mathrm{ord}_M(g)=k_0$, we can infer that for some $j$, $\mathrm{ord}_M(f_j)=k_0$. 
This proves our first statement.

$(1)\Rightarrow(2)$: For a fixed $p\in M$, we may assume that $f_1$ generates $(f_1,\cdots,f_N)_p$, namely 
$\frac{f_2}{f_1},\cdots,\frac{f_n}{f_1}$ are holomorphic near $p$, 
therefore, $M=f_1^{-1}(0)\cap\cdots\cap f_N^{-1}(0)=f_1^{-1}(0)$ near $p$. 
As a result, $f_1=\lambda^{k_0}h$ near $p$, where $h$ is non-vanishing and $\lambda$ is a coordinate function such that $\{\lambda=0\}=M$ near $p$. 
We expand $f_1$ into Taylor series:
\begin{align*}
f_1=\sum_{k=k_0}f_{1,k},\ \ f_{1,k}\in H^0(M,L^{k}).
\end{align*}
Clearly, $f_{1,k_0}$ does not vanish at $p$, hence $(2)$ follows.

$(2)\Rightarrow(1)$: For a fixed $p\in M$, there is some $g\in H^0(M,L^{k_0})$ such that $g(p)\neq0$. 
As before, we have $g=\sum_{\alpha}c_{\alpha}f_1^{\alpha_1}\cdot...\cdot f_N^{\alpha_N}$. 
This implies that for some $j$, $f_j=\lambda^{k_0}h$ near $p$, $h(p)\neq0$. 
Therefore, $f_{j}$ generates $(f_1,\cdots,f_N)_p$ and $(1)$ follows as desired.
\end{proof}

We now give the proof of Theorem \ref{thm:vol control}.
\begin{proof}[Proof of Theorem \ref{thm:vol control}.]
Let $f:(L^*,M)\ra (Z,z_0)$ be the normal Grauert blow-down. 
We select $g_1,\cdots,g_N\in H^0(M,L^{k_1})$ such that they generate the fibers of $L^{k_1}$ at every point of $M$. 
Regarding $g_i$ as holomorphic functions on $L^*$, we see that $\{g_1=\cdots=g_N=0\}=M$.
Since $Z$ is normal and $f$ induces a biholomorphic map from $L^*\backslash M$  to $X\backslash\{z_0\}$,
$g_1,\cdots, g_N$ induce holomorphic functions, say $\tilde g_1,\cdots, \tilde g_N$, on $Z$, whose common zero set is the single point $z_0$.
Therefore $\mathfrak{U}:=( \tilde g_1,\cdots, \tilde g_N)$ is an $\mathfrak{M}$-primary ideal in $\mathcal{O}_{Z,z_0}$. Clearly, the vanishing order of $f^*\mathfrak{U}$ along $M$ is $k_1$. 
The proof of the volume inequality will involve the mixed multiplicities $e(\mathfrak{U}^{[n-p]};\mathfrak{M}^{[p+1]}), e(\mathfrak{U}^{[n+1-p]};\mathfrak{M}^{[p]})$.

On an open subset $U\subseteq M$ such that $L^*|_{U}\cong U_z\times\mathbb{C}_{\lambda}$, 
we have $g_i(z,\lambda)=\lambda^{k_1}\varphi_i(z)$ and $\varphi_i$ does not vanish simultaneously on $U$, therefore
\begin{align*}
\log|g|=\frac{1}{2}\log(|g_1|^2+\cdots+|g_N|^2)=k_1\log|\lambda|+\frac{1}{2}\log(|\varphi_1|^2+\cdots+|\varphi_N|^2)
\end{align*}
and
\begin{align*}
dd^c\log|g|=k_1[M]+\theta,\ \theta=\frac{1}{2}dd^c\log(|\varphi_1|^2+\cdots+|\varphi_N|^2).
\end{align*}
Similarly, if $\mathfrak{M}$ is generated by $z_1,\cdots,z_{M}$, for $|f|=(|f^*z_1|^2+\cdots+|f^*z_{M}|)^{\frac{1}{2}}$, we have
\begin{align*}
dd^c\log|f|=k_0[M]+T, \ T=dd^c\log\big(\big|\frac{f}{\lambda^{k_0}}\big|\big).
\end{align*}
Since the singularities of currents $(\theta^{[n+1-p]},T^{[p]})$, $(\theta^{[n-p]},T^{[p+1]})$ satisfy the dimension conditions in Theorem \ref{thm:main thm},
we can derive the explicit fomula for multiplicities $e(\mathfrak{U}^{[n-p]};\mathfrak{M}^{[p+1]})$ and $e(\mathfrak{U}^{[n+1-p]};\mathfrak{M}^{[p]})$ as follows:
\begin{align*}
e(\mathfrak{U}^{[n+1-p]};\mathfrak{M}^{[p]})&=(-1)^n(k_1M)^{n+1-p}\cdot(k_0M)^p+\int_M \theta^{n+1-p}\wedge T^{p}\\
&=k_1^{n+1-p}k_0^{p}\int_Mc_1(L)^n+\int_M\theta^{n+1-p}\wedge T^p,\\
e(\mathfrak{U}^{[n-p]};\mathfrak{M}^{[p+1]})&=(-1)^n(k_1M)^{n-p}\cdot(k_0M)^{p+1}+\int_M \theta^{n-p}\wedge T^{p+1}\\
&=k_0^{n-p}k_1^{p+1}\int_Mc_1(L)^n+\int_M\theta^{n-p}\wedge T^{p+1}.
\end{align*}
Note that $T$ is smooth on $L^*\setminus B$, $\theta^{n+1-p}\wedge T^{p}$ puts no mass on $M\setminus B$, so we have
\begin{align*}
\int_M\theta^{n+1-p}\wedge T^p=\int_B\theta^{n+1-p}\wedge T^p=\int_{L^*}\theta^{n+1-p}\wedge\mathds{1}_B T^{p}.
\end{align*}
$T^p$ is a closed positive current of bidimension $(n+1-p,n+1-p)$ on $L^*$, the support theorem for closed current implies that $\mathds{1}_BT^p=0$. Consequently, we obtain that
\begin{align}
e(\mathfrak{U}^{[n+1-p]};\mathfrak{M}^{[p]})=k_0^pk_1^{n+1-p}\int_Mc_1(L)^n.
\end{align}
Similarly, we have
\begin{align*}
\int_M\theta^{n-p}\wedge T^{p+1}=\int_{L^*}\theta^{n-p}\wedge\mathds{1}_B T^{p+1}.
\end{align*}
Note that $\mathds{1}_B T^{p+1}=\sum\lambda_j[B_j]$, 
where $\lambda_j\in\mathbb{N}^*$ and $B_j$ is the irreducible component of $B$, $\mathrm{codim}_M B_j=p$ (see notations and notions ahead of Theorem \ref{thm:vol control} in Introduction).
Note that $\{\theta|_M\}=c_1(k_1L)$, so we obtain that
\begin{align}
e(\mathfrak{U}^{[n-p]};\mathfrak{M}^{[p+1]})=k_0^{p+1}k_1^{n-p}\int_Mc_1(L)^n+k_1^{n-p}\sum_{j}\lambda_j\int_{B_j}c_1(L)^{n-p}.
\end{align}\label{eq:eq for mixed mult and ch class}
Since $\mathfrak{U}\subseteq \mathfrak{M}$, the inequality $e(\mathfrak{U}^{[n-p]};\mathfrak{M}^{[p+1]})\leq e(\mathfrak{U}^{[n+1-p]};\mathfrak{M}^{[p]})$ then yields that
\begin{align*}
\sum_{j}\lambda_j\int_{B_j}c_1(L)^{n-p}\leq (k_0^pk_1-k_0^{p+1})\int_Mc_1(L)^n.
\end{align*}
The proof is complete.
\end{proof}


Now we can complete the proof of Theorem \ref{thm:mult of (Z,z_0)}. We need Teissier's inequality for mixed multiplicities \cite{Tei77},\cite{RS78}:
\begin{lem}[Rees-Sharp]\label{lem:Rees-Sharp inequality}
Let $(\mathcal{O},\mathfrak{M})$ be a local commutative Noetherian ring of dimension $d\geq2$. 
Suppose $\mathfrak{U},\mathfrak{V}$ are two $\mathfrak{M}$-primary ideals in $\mathcal{O}$, then for $i=1,\cdots,d-1$, it holds that
\begin{align*}
e(\mathfrak{U}^{[i]};\mathfrak{V}^{[d-i]})^2 \leq e(\mathfrak{U}^{[i-1]};\mathfrak{V}^{[d-i+1]})\cdot e(\mathfrak{U}^{[i+1]};\mathfrak{V}^{[d-i-1]}).
\end{align*}
\end{lem}

\begin{thm}(=Theorem \ref{thm:mult of (Z,z_0)}) 
Let $M$ be a compact complex manifold of dimension $n$, and $L$ be an ample line bundle over $M$. 
Suppose $(Z,z_0)$ is the normal Grauert blow-down of the dual bundle $L^*$, it holds that
\begin{align*}
&\mathrm{mult}(Z,z_0)\geq k_0^{n+1}\mathrm{vol}(L)+(n+1-p)k_0^{n-p}\mathrm{vol}_B(L),\\
&\mathrm{mult}(Z,z_0)\leq k_0^{p+1}k_1^{n-p}\mathrm{vol}(L)+k_1^{n-p}\mathrm{vol}_B(L).
\end{align*}

\end{thm}

\begin{proof}
We use the notations in the proof of Theorem \ref{thm:vol control}.

If $L^{k_0}$ is globally generated, $f^*\mathfrak{M}$ is invertible by Proposition \ref{prop:L globally generated}, Theorem \ref{thm:main thm} then implies that
\begin{align*}
\mathrm{mult}(Z,z_0)=e(\mathfrak{M}^{[n+1]})=(-1)^n(k_0M)^{n+1}=k_0^{n+1}\mathrm{vol}(L).
\end{align*}
If $L^{k_0}$ is not globally generated, i.e. $B$ is non-empty, Lemma \ref{lem:Rees-Sharp inequality} implies that for $i=1,\cdots,n$,
\begin{align*}
e(\mathfrak{U}^{[n+1-i]};\mathfrak{M}^{[i]})^2\leq e(\mathfrak{U}^{[n+2-i]};\mathfrak{M}^{[i-1]})\cdot e(\mathfrak{U}^{[n-i]};\mathfrak{M}^{[i+1]}).
\end{align*}
In particular, $\frac{e(\mathfrak{M}^{[n+1]})}{e(\mathfrak{U}^{[1]},\mathfrak{M}^{[n]})} \geq \frac{e(\mathfrak{U}^{[1]},\mathfrak{M}^{[n]})}{e(\mathfrak{U}^{[2]},\mathfrak{M}^{[n-1]})} \geq \cdots \geq \frac{e(\mathfrak{U}^{[n-p]},\mathfrak{M}^{[p+1]})}{e(\mathfrak{U}^{[n+1-p]},\mathfrak{M}^{[p]})}$, so we obtain that
\begin{align*}
\mathrm{mult}(Z,z_0)=e(\mathfrak{M}^{[n+1]})
&= \prod_{i=1}^{n+1-p}\frac{e(\mathfrak{U}^{[i-1]};\mathfrak{M}^{[n+i]})}{e(\mathfrak{U}^{[i]};\mathfrak{M}^{[n+1-i]})} \cdot e(\mathfrak{U}^{[n+1-p]};\mathfrak{M}^{[p]})\\
&\geq \Big(\frac{e(\mathfrak{U}^{[n-p]};\mathfrak{M}^{[p+1]})}{e(\mathfrak{U}^{[n+1-p]};\mathfrak{M}^{[p]})}\Big)^{n+1-p}\cdot e(\mathfrak{U}^{[n+1-p]};\mathfrak{M}^{[p]}).
\end{align*}
By the proof of Theorem \ref{thm:vol control},
\begin{align*}
\frac{e(\mathfrak{U}^{[n-p]};\mathfrak{M}^{[p+1]})}{e(\mathfrak{U}^{[n+1-p]};\mathfrak{M}^{[p]})}&=\frac{k_0}{k_1}+\frac{\mathrm{vol}_B(L)}{k_0^pk_1\mathrm{vol}(L)},\\
e(\mathfrak{U}^{[n+1-p]};\mathfrak{M}^{[p]})&=k_0^pk_1^{n+1-p}\mathrm{vol}(L),
\end{align*}
we can infer that
\begin{align*}
\mathrm{mult}(Z,z_0)&\geq \Big(\frac{k_0}{k_1}+\frac{\mathrm{vol}_B(L)}{k_0^pk_1\mathrm{vol}(L)}\Big)^{n+1-p}\cdot \Big( k_0^pk_1^{n+1-p}\mathrm{vol}(L)\Big)\\
&\geq k_0^{n+1}\mathrm{vol}(L)+(n+1-p)k_0^{n-p}\mathrm{vol}_B(L).
\end{align*}
So we get the first inequality in the theorem.

The second inequality follows from \eqref{eq:eq for mixed mult and ch class} in the proof of Theorem \ref{thm:vol control}, and the formula 
$$\mathrm{mult}(Z,z_0)=e(\mathfrak{M}^{[n+1]})\leq e(\mathfrak{U}^{[n-p]};\mathfrak{M}^{[p+1]}).$$
\end{proof}

We can now deduce Corollary \ref{cor:g.g case} from Theorem \ref{thm:mult of (Z,z_0)}. 
\begin{thm}(= Corollary \ref{cor:g.g case})
Let $M$ be a compact manifold of dimension $n$, $E$ be an ample vector bundle over $M$. 
Suppose $(Z,z_0)$ is the normal Grauert blow-down of the dual bundle $E^*$, then $E$ is globally generated if and only if
\begin{align*}
\mathrm{mult}(Z,z_0)=\int_Ms_n(E^*).
\end{align*}

\end{thm}

\begin{proof}[Proof of Corollary \ref{cor:g.g case}.]
We reduce general cases to the case that $E$ is a line bundle.
Let $\mathcal{O}_{\mathbf{P}(E^*)}(-1)$ be the tautological line bundle over the projectivation bundle
\begin{align*}
\textbf{P}(E^*):=(E^* - M)/\mathbb{C}^*,
\end{align*}
It is obvious that the Grauert blow down $f:E^*\rightarrow Z$ induces the Grauert blow down $\textbf{f}:\mathcal{O}_{\mathbf{P}(E^*)}(-1)\rightarrow Z$ in the following manner:
\begin{displaymath}
\xymatrix{
\mathcal{O}_{\mathbf{P}(E^*)}(-1) \ar[r]^{\ \ \ \ \ \pi}\ar@{-->}[dr]_{\textbf{f}:=f\circ \pi} &E^* 
\ar[d]^{f} \\
&Z 
}.
\end{displaymath}
There is a natural isomorphism
\begin{align*}
H^0(M,S^kE)\cong H^0(\mathbf{P}(E^*),\mathcal{O}_{\mathbf{P}(E^*)}(k))
\end{align*}
for $k\geq 1$, where $S^kE$ is the $k$-th symmetric power of $E$.
Moreover, $S^kE$ is globally generated if and only if $\mathcal{O}_{\mathbf{P}(E^*)}(k)$ is globally generated. 


Using the well-known fomula (see e.g. \cite[Lemma A.1]{Gr66}, \cite[Proposition 1.1]{Di16}):
\begin{align*}
\pi_{*}(c_1(\mathcal{O}_{\mathbf{P}(E^*)}(1))^{k+r-1})=s_{k}(E^*),\ k=0,1,\cdots,
\end{align*}
we obtain the fomula
\begin{align*}
\int_{\textbf{P}(E^*)}c_1(\mathcal{O}_{\mathbf{P}(E^*)}(1))^{n+r-1}=\int_{M}s_n(E^*).
\end{align*}

The above consideration reduces reduce the corollary to the case that $E$ is a line bundle, say $L$.
In this case, if $L$ is globally generated, we have $k_0=k_1=1$ and $\mathrm{vol}_B(L)=0$, Theorem \ref{thm:mult of (Z,z_0)} implies that $\mathrm{mult}(Z,z_0)=\int_Mc_1(L)^n$.
Conversely, the lower bound for multiplicity 
\begin{align*}
\mathrm{mult}(Z,z_0)\geq k_0^{n+1}\int_Mc_1(L)^n+(n-p+1)\mathrm{vol}_B(L)
\end{align*}
first implies that $k_0=1$. 
If the base locus of $L$ is non-empty, we have
\begin{align*}
\mathrm{vol}_B(L)=\sum\lambda_j\int_{B_j}c_1(L)^{n-p}>0
\end{align*}
thanks to $c_1(L)>0$. The inequality above yields that $\mathrm{mult}(Z,z_0)>\int_Mc_1(L)^n$, a contradiction. 
This proves that $L$ is globally generated. 
\end{proof}

\begin{problem}
Suppose $X$ is a complex manifold of dimension $n$, $E\subseteq X$ is an irreducible hypersurface which is exceptional, that is, there is a modification $f:(X,E)\ra (Z,z_0)$ such that $Z$ is a normal complex space with isolated singularity $z_0$.

Let $\mathcal{I}:=f^*\mathfrak{M}$ be the pull-back of the maximal ideal $\mathcal{M}$ in $\mathcal{O}_{Z,z_0}$, $k$ be the vanishing order of $\mathcal{I}$ along $E$.
Are the following two statements equivalent?
\begin{itemize}
\item
$\mathcal{I}$ is invertible,
\item
$\mathrm{mult}(Z,z_0)=(-1)^{n-1}(kE)^n$.
\end{itemize}
\end{problem}

In the special case that $M$ being a compact Riemann surface, we obtain the explicit fomula of $\mathrm{mult}(Z,z_0)$.
\begin{thm}(=Theorem \ref{thm:RS})
Let $M$ be a compact Riemann surface and $L$ be an ample line bundle on $M$. We introduce the following notations:
\begin{itemize}
\item[$\bullet$]
$k_0$: the minimal integer such that $L^{k_0}$ has non-zero global section,
\item[$\bullet$]
$P_1,\cdots,P_N$: base points of $L^{k_0}$,
\item[$\bullet$]
$k_1,\cdots,k_N$: the minimal integer such that $L^{k_j}$ generates $P_j$.
\item[$\bullet$]
$d_{j,k}$: the vanishing order of $H^0(M,L^{k})$ (viewed as holomorphic functions on $L^*$) at $P_j$, $j=1,\cdots,N$, $k=k_0,\cdots,k_j$.
\end{itemize}
Suppose that $(Z,z_0)$ is the normal Grauert blow-down of $L^*$, then
\begin{align*}
\mathrm{mult}(Z,z_0)=k_0^2\mathrm{deg}(L)+\sum_{j=1}^N\lambda_j,
\end{align*}
where $\lambda_j$ is the multiplicity of ideal generated by $\{z^{d_{j,k}}w^{k-k_0}:k=k_0,\cdots,k_j\}$ in $\mathcal{O}_{\mathbb{C}_{z,w}^2,0}$.
\end{thm}

\begin{proof}
Suppose that $f:(L^*,M)\ra(Z,z_0)$ is the normal Grauert blow-down and $(Z,z_0)\subseteq(\mathbb{C}^N_z,0)$. 
We set $f_i=f^*z_i$, $i=1,\cdots,N$, Corollary \ref{cor:two dim} yields that
\begin{align*}
\mathrm{mult}(Z,z_0)=e(\mathfrak{M}_{Z,z_0})&=-(k_0M)\cdot(k_0M)+\int_MT^2\\
&=k_0^2\mathrm{deg}(L)+\sum_{j=1}^N \int_{\{P_j\}} T^2,
\end{align*}
where $T=dd^c\log|f|-k_0[M]$.

We take neighbourhood $U_j$ of $P_j$ such that $L^*|_{U_j}\cong (U_j)_z\times\mathbb{C}_{\lambda}$
and claim that
\begin{align*}
\log|f|= 
 \log\big(\sum_{k=k_0}^{k_j}|z^{d_{j,k}}\lambda^{k}|\big)+O(1)\ \mathrm{near }\ P_j. \tag{*}
\end{align*}
We take $g_{j,k}\in H^0(M,L^{k})$ such that $\mathrm{ord}_{P_j}(g_{j,k})=d_{j,k}$, $k=k_0,\cdots,k_j$.
Viewing $g_{j,s}$ as holomorphic functions on $L^*$, we get that 
\begin{align*}
g_{j,k}(z,\lambda)=\lambda^{k}\cdot \varphi_k(z),\ \mathrm{ord}_{P_j}(\varphi_k)=d_{j,k}.
\end{align*}
From the expression $g_{j,k}=\sum c_{\alpha,k}f_1^{\alpha_1}\cdot...\cdot f_N^{\alpha_N}$, we get that $|g_{j,k}|\leq C(|f_1|+\cdots+|f_N|)$ near $P_j$. It follows that
\begin{align*}
\log|f| \geq \log\big(\sum_{k=k_0}^{k_j}|g_{j,k}| \big)+O(1)=\log\big(\sum_{k=k_0}^{k_j}|z^{d_{j,k}}\lambda^{k}| \big)+O(1).
\end{align*}
To prove the inverse inequality, we expand $f_i$ near $P_j$
\begin{align*}
f_i=\sum_{k=k_0}f_{i,k},\ f_{i,k}\in H^0(M,L^{k}).
\end{align*}
Our assumption yields that
\begin{align*}
&|f_{i,k}|\leq C|\lambda^{k}z^{d_{j,k}}|,\ k=k_0,\cdots,k_j,\\
&|f_{i,k}|\leq C|\lambda^{k_j}|,\ k> k_j.
\end{align*}
Therefore,
\begin{align*}
\log|f|&=\frac{1}{2}\log(|f_1|^2+\cdots+|f_N|^2)\\
&\leq \log \big(\sum_{k=k_0}^{k_j}|z^{d_{j,k}}\lambda^{k}| \big)+O(1).
\end{align*}
This completes the proof of $(*)$. From $T=dd^c\log|f|-k_0[M]=dd^c\log\big(\frac{{|f|}}{|\lambda|^{k_0}}\big)$, we use the comparison theorem of Lelong number (see Lemma \ref{lem:comparison thm}) to obtain
\begin{align*}
\int_{\{P_j\}}T^2=\int_{\{0\}}\Big(dd^c\log\big(\sum_{k=k_0}^{k_j}|z^{d_{j,k}}\lambda^{k-k_0}| \big)\Big)^2.
\end{align*}
Moreover, Theorem \ref{thm:mult vs lelong} yields that
\begin{align*}
\int_{\{0\}}\Big(dd^c\log\big(\sum_{k=k_0}^{k_j}|z^{d_{j,k}}\lambda^{k-k_0}| \big)\Big)^2=e\big((z^{d_{j,k_0}},\cdots,\lambda^{k_j-k_0})\big).
:=\lambda_j 
\end{align*}
So multiplicity fomula
\begin{align*}
\mathrm{mult}(Z,z_0)=k_0^2\mathrm{deg}(L)+\sum_{j=1}^N\lambda_j
\end{align*}
follows as desired.
\end{proof}

We are no in a position to prove Corollary \ref{cor:Weierstrass point}.

\begin{cor}(=Corollary \ref{cor:Weierstrass point})
Let $M$ be a compact Riemann surface of genus $g$, $P\in M$. Suppose $(Z,z_0)$ is the normal Grauert blow-down of $L_{-P}$, the line bundle generated by the divisor $-P$, then
\begin{align*} 
\mathrm{mult}(Z,z_0)=\text{the\ first\ nongap\ of\ $M$\ at\ $p$}.
\end{align*}
In particular, if $P$ is not a Weierstrass point of $M$, then
\begin{align*}
\mathrm{mult}(Z,z_0)=g+1.
\end{align*}
\end{cor}
\begin{proof}
Let $L=L_P=(L_{-P})^*$ be the dual bundle, we have $k_0=1$, $\{P_1,\cdots,P_N\}=\{P\}$, $d_{1,1}=1$ and $k_1$ coincides with the minimal integer $k$ which is not a gap of $M$ at $P$. Theorem \ref{thm:RS} then implies that
\begin{align*}
\mathrm{mult}(Z,z_0)=\mathrm{deg}(L)+(k-1)=k.
\end{align*}
If $P$ is not a Weierstrass point of $M$, the minimal nongap of $M$ at $P$ is $g+1$, this proves the second statement.
\end{proof}

	\end{document}